\documentclass[reqno,b5paper]{amsart}
\usepackage{amsmath}
\usepackage{amssymb}
\usepackage{amsthm}
\usepackage{enumerate}
\usepackage[mathscr]{eucal}
\theoremstyle{plain}
\newtheorem{thm}{Theorem}[section]
\newtheorem{cor}[thm]{Corollary}

\newtheorem{note}{Note}[section]

\theoremstyle{definition}
\newtheorem{defn}{Definition}[section]

\begin{document}

\setcounter {page}{1}
\title{ Generalized statistical limit points and statistical cluster points via ideal}

\author[P. Malik and AR. Ghosh]{ Prasanta Malik* and Argha Ghosh*\ }
\newcommand{\acr}{\newline\indent}
\maketitle
\address{{*\,} Department of Mathematics, The University of Burdwan, Golapbag, Burdwan-713104,
West Bengal, India.
                Email: pmjupm@yahoo.co.in. , papanargha@gmail.com \acr
           }

\maketitle
\begin{abstract}
In this paper we have extended the notion of statistical limit point as introduced by Fridy\cite{Fr2} to $I$-statistical limit point of sequences of real numbers and studied some basic properties of the set of all $I$-statistical limit points and $I$-statistical cluster points of real sequences. 
\end{abstract}
\author{}
\maketitle


\section{\textbf{Introduction:}}
 The usual notion of convergence of real sequences was extended to statistical convergence independently by Fast\cite{Fa} and Schoenberg\cite{Sc} based on the notion of natural density of subsets of $\mathbb N$, the set of all positive integers. Since then a lot of works has been done in this area (in particular after the seminal works of Salat\cite{Sa} and Fridy\cite{Fr1}). Following the notion of statistical convergence in \cite{Fr2} Fridy introduced and studied the notions of statistical limit points and statistical cluster points of real sequences. More primary work on this convergence can be found from\cite{Co1, Co2, Co3, Fr3, Pe, St}  where other references can be found.
 
The concept of statistical convergence was further extended to $I$-convergence by Kostyrko et. al.\cite{Ko1} using the notion of ideals of $\mathbb N$.  Using this notion of ideals the concepts of statistical limit point and statistical cluster point were naturally extended to $I$-limit point and $I$-cluster point respectively by Kostyrko et. al. in \cite{Ko2}. More works in this line can be found in \cite{De, La1, La2} and many others.

Recently in \cite{Da1} the notion of $I$-statistical convergence and $I$-statistical cluster point of real sequences have been introduced by Das et. al. using ideals of $\mathbb N$, which naturally extends the notions of statistical convergence and statistical cluster point. Further works on such summability method can be found in \cite{ Da2, Mu, Da3} and many others.

In this paper using the notion of $I$-statistical convergence we extend the concept of statistical limit point to $I$-statistical limit point of sequences of real numbers and study some properties of $I$-statistical  limit points and $I$-statistical cluster points of sequences of real numbers. We also study the sets of $I$-statistical limit points and $I$-statistical cluster points of sequences of real numbers and relationship between them.

\section{\textbf{Basic Definitions and Notations}}
In this section we recall some basic definitions and notations.
 Throughout the paper $\mathbb N$ denotes the set of all positive integers, $\mathbb R$ denotes the set of all real numbers and $x$  denotes the sequence $\{x_k\}_{k\in\mathbb N}$ of real numbers.
\begin{defn}\cite{Ne}
A subset $K$ of $\mathbb N$ is said to have natural density (or asymptotic density) $d(K)$ if
 \begin{center}
 $d(K)=\underset{n\rightarrow \infty}{\lim}\frac{\left|K(n)\right|}{n}$ 
 \end{center}
 where $K(n)=\left\{j\in K:j\leq n\right\}$ and $\left|K(n)\right|$ represents the number of elements in $K(n)$.
\end{defn}
If $\{x_{k_j}\}_{j\in\mathbb N}$ is a subsequence of the sequence $ x = \{x_k\}_{k\in\mathbb N}$ of real numbers and
$A =\{k_j:j\in\mathbb N\}$, then we abbreviate $\{x_{k_j}\}_{j\in\mathbb N}$  by $\{x\}_A$. In case $d(A) = 0$,
$\{x\}_A$ is called a subsequence of natural density zero or a thin subsequence of $x$. On the other
hand, $\{x\}_A$ is a non-thin subsequence of $x$ if $d(A)$ does not have natural density zero i.e., if either $d(A)$  is
a positive number or $A$ fails to have natural density.
\begin{defn}\cite{Fr1}
Let $x=\{x_k\}_{k \in \mathbb{N}}$ be a sequence of real numbers.
Then $x$ is said to be statistically convergent to $ \xi $
if for any $\varepsilon > 0$
\begin{center}
$d(\{ k : \left| x_k - \xi
\right| \geq  \varepsilon\})=0 $.
\end{center}
In this case we write $ st-\underset{k\rightarrow \infty}{\lim}~ x_k = \xi $.
\end{defn}
\begin{defn}\cite{Pe}
A sequence $ x=\left\{x_{k}\right\}_{k\in\mathbb N}$ is said to be statistically bounded if there exists a compact set $C$ in $\mathbb R$ such that  $d(\{ k:k \leq n, x_{k}\notin C \})=0.$
\end{defn}

\begin{defn}\cite{Fr2}
A real number $l$ is a statistical limit point of the sequence
$x=\{x_k\}_{k\in\mathbb N}$ of real numbers, if there exists a nonthin subsequence of $x$ that converges to $l$. 
\end{defn}
A real number $L$ is an ordinary limit point of a sequence $x$ if there is a subsequence of $x$ that converges to $L$. The set of all ordinary limit points and the set of all statistical limit points of the sequence $x$ are denoted by $L_x$ and $\Lambda_x$ respectively. Clearly $\Lambda_x\subset L_x$.
\begin{defn}\cite{Fr2}
 A real number $y$ is a statistical cluster point of the sequence
$x=\{x_k\}_{k\in\mathbb N}$ of real numbers, if for any $\varepsilon > 0$ the set $\{k \in\mathbb N : \left|x_k - y\right| <\varepsilon\}$ does not have natural density zero. 
\end{defn}
The set of all statistical cluster points of $x$ is denoted by $\Gamma_x$. Clearly $\Gamma_x\subset L_x$.

We now recall definitions of ideal and filter in a non-empty set.
\begin{defn}\cite{Ko1}
 Let $X\neq\phi $. A class $ I $ of subsets of $X$ is said to be
an ideal in X provided, $I$ satisfies the conditions:
\\(i)$\phi \in I$,
\\(ii)$ A,B \in I \Rightarrow A \cup B\in I,$
\\(iii)$ A \in I, B\subset A \Rightarrow B\in I$.
\end{defn}

An ideal $I$ in a non-empty set $X$ is called non-trivial if $ X \notin I.$
\begin{defn}\cite{Ko1}
 Let $X\neq\phi  $. A non-empty class $\mathbb F $ of subsets of $X$ is
said to be a filter in $X$ provided that:
\\(i)$\phi\notin \mathbb F $,
\\(ii) $A,B\in\mathbb F \Rightarrow A \cap B\in\mathbb F,$
\\(iii)$ A \in\mathbb F, B\supset A \Rightarrow B\in\mathbb F$.
\end{defn}
\begin{defn}\cite{Ko1}
 Let $I$ be a non-trivial ideal in a non-empty set $ X$.
 Then the class $\mathbb F(I)$$ = \left\{M\subset X : \exists A \in I ~such~~ that ~M = X\setminus A\right\}$
is a filter on $X$. This filter $\mathbb F(I) $ is called the filter associated with $I$.
\end{defn}
 A non-trivial ideal $I$ in $X(\neq \phi)$ is called admissible if $\left\{x\right\} \in I$ for each $x \in X$.
Throughout the paper we take $I$ as a non-trivial admissible ideal in $\mathbb{N} $ unless otherwise mentioned.
\begin{defn}\cite{Ko1}
Let $x=\{x_k\}_{k \in \mathbb{N}}$ be a sequence of real numbers.
Then $x$ is said to be $I$-convergent to $ \xi $
if for any $\varepsilon > 0$
\begin{center}
$\{ k : \left| x_k - \xi
\right| \geq  \varepsilon\} \in I $.
\end{center}
In this case we write $ I-\underset{k\rightarrow \infty}{\lim}~ x_k = \xi $.
\end{defn}
\begin{defn}\cite{De}
A sequence $x=\{x_k\}_{k\in\mathbb N}$ of real number is said to be $I$-bounded if there exists a real number $G>0$ such that $\{k:\left|x_k\right|>G\}\in I$.
\end{defn}

\begin{defn}\cite{Da1}
Let $x=\{x_k\}_{k \in \mathbb{N}}$ be a sequence of real numbers.
Then $x$ is said to be
 $I$-statistically convergent to $ \xi $
if for any $\varepsilon > 0$ and $ \delta > 0$
\begin{center}
$\{ n \in \mathbb{N}: \frac{1}{n}|\{ k \leq n: \left| x_k - \xi
\right| \geq  \varepsilon\}| \geq \delta \} \in I $.
\end{center}
In this case we write $ {I\mbox{-}st\mbox{-}\lim}~ x = \xi $.
\end{defn}


\section{\textbf{$I$-statistical limit points and $I$-statistical cluster points}}
In this section, following the line of Fridy \cite{Fr2} and Pehlivan et. al. \cite{Pe}, we introduce the notion of $I$-statistical limit point of real sequences and present an $I$-statistical analogue of some results in those papers.
\begin{defn}\cite{Da2}
A subset $K$ of $\mathbb N$ is said to have $I$-natural density (or, $I$-asymptotic density)
$d_I(K)$ if
 \begin{center}
 $d_I(K)=I-\underset{n\rightarrow \infty}{\lim}\frac{\left|K(n)\right|}{n}$ 
 \end{center}
 where $K(n)=\left\{j\in K:j\leq n\right\}$ and $\left|K(n)\right|$ represents the number of elements in $K(n)$.
\end{defn}
\begin{note}
From the above definition, it is clear that, if $d(A) = r, A \subset \mathbb{N}$, then $d_I (A) = r $ 
for any nontrivial admissible ideal $I$ in $\mathbb{N}$. 
\end{note}
In case $d_I(A) = 0$,
$\{x\}_A$ is called a subsequence of $I$-asymptotic density zero, or a $I$-thin subsequence of $x$. On the other
hand, $\{x\}_A$ is an $I$-nonthin subsequence of $x$, if $d_I(A)$ does not have density zero i. e., if either $d_I(A)$  is
a positive number or $A$ fails to have $I$-asymptotic density.

\begin{defn}
 A real number $l$ is an $I$-statistical limit point of a sequence
$x=\{x_k\}_{k\in\mathbb N}$ of real numbers, if there exists an $I$-nonthin subsequence of $x$ that converges to $l$. The set of all $I$-statistical limit points of the sequence $x$ is denoted by $\Lambda_x^S(I)$. 
\end{defn}
\begin{defn}\cite{Da2}
 A real number $y$ is an $I$-statistical cluster point of a sequence
$x=\{x_k\}_{k\in\mathbb N}$ of real numbers, if for every $\varepsilon > 0$, the set $\{k \in\mathbb N : \left|x_k - y\right| <\varepsilon\}$ does not have $I$-asymptotic density zero. The set of all $I$-statistical cluster points of $x$ is denoted by $\Gamma_x^S(I)$.
\end{defn}
\begin{note}
If $I=I_{fin}=\{A\subset \mathbb N: \left|A\right|<\infty\}$, then the notions of $I$-statistical limit points and $I$-statistical cluster points coincide with the notions of statistical limit points and statistical cluster points respectively.
\end{note}

We first present an $I$-statistical analogous of some results in \cite{Fr2}.
\begin{thm}
Let $x=\{x_k\}_{k\in\mathbb N}$ be a sequence of real numbers. Then $\Lambda_x^S(I)\subset \Gamma_x^S(I)$.
\end{thm}
\begin{proof}
Let $\alpha\in \Lambda_x^S(I) $. So we have a subsequence $\{x_{k_j}\}_{j\in\mathbb N}$ of $x$ with $lim~ x_{k_j}=\alpha$ and $d_I(K)\neq 0$, where $K=\{k_j:j\in\mathbb N\}$. Let $\varepsilon>0$ be given. Since $lim~ x_{k_j}=\alpha$, so $ B = \{k_j:\left|x_{k_j}-\alpha\right|\geq \varepsilon\}$ is a finite set. Thus 
\begin{center}
$\{k\in\mathbb N:\left|x_k-\alpha\right|<\varepsilon\}\supset \{k_j:j\in\mathbb N\}\setminus B$
\end{center}
\begin{center}
$\Rightarrow K = \{k_j:j\in\mathbb N\}\subset \{k\in\mathbb N:\left|x_k-\alpha\right|<\varepsilon\}\cup B$.
\end{center}
Now if $d_I(\{k\in\mathbb N:\left|x_k-\alpha\right|<\varepsilon\})=0$, then we have $d_I(K)=0$, which is a contradiction. Thus $\alpha$ is an $I$-statistical cluster point of $x$. Since $\alpha\in \Lambda_x^S(I)$ is arbitrary, so $\Lambda_x^S(I)\subset \Gamma_x^S(I)$.
\end{proof}
\begin{note}
The set $\Lambda_x^S(I)$ of all $I$-statistical limit points of a sequence $x$ may not be equal to the set $\Gamma_x^S(I)$ of all $I$-statistical cluster points of $x$. To show this we cite the following example.
 \end{note}
\begin{thm}
If $x=\{x_k\}_{k\in\mathbb N}$ and $y=\{y_k\}_{k\in\mathbb N}$ are sequences of real numbers such that $d_I(\{k:x_k\neq y_k\})=0$, then $\Lambda_x^S(I)=\Lambda_y^S(I)$ and $\Gamma_x^S(I)=\Gamma_y^S(I)$.
\end{thm}
\begin{proof}
Let $\gamma\in \Gamma_x^S(I)$ and $\varepsilon>0$ be given. Then $\{k:\left|x_k-\gamma\right|<\varepsilon\}$ does not have $I$-asymptotic density zero. Let $A	=\{k:x_k=y_k\}$. Since $d_I(A)=1$ so $\{k:\left|x_k-\gamma\right|<\varepsilon\}\cap A$ does not have $I$-asymptotic density zero. Thus $\gamma\in \Gamma_y^S(I)$. Since $\gamma\in \Gamma_x^S(I)$ is arbitrary, so $\Gamma_x^S(I)\subset \Gamma_y^S(I)$. By symmetry we have $\Gamma_y^S(I)\subset \Gamma_x^S(I)$. Hence $\Gamma_x^S(I)=\Gamma_y^S(I)$.

Also let $\beta\in \Lambda_x^S(I) $. Then $x$ has an $I$-nonthin subsequence $\{x_{k_j}\}_{j\in\mathbb N}$ that converges to $\beta$. Let $K=\{k_j:j\in\mathbb N\}$. Since $d_I(\{k_j: x_{k_j}\neq y_{k_j}\})=0$, we have $d_I(\{k_j:x_{k_j}=y_{k_j}\})\neq 0$. Therefore from the latter set we have an $I$-nonthin subsequence $\{y\}_{K'}$ of $\{y\}_K$ that converges to $\beta$. Thus $\beta\in \Lambda_y^S(I)$. Since $\beta\in\Lambda_x^S(I)$ is arbitrary, so $\Lambda_x^S(I)\subset \Lambda_y^S(I)$. By symmetry we have $\Lambda_x^S(I)\supset \Lambda_y^S(I)$. Hence $\Lambda_x^S(I)=\Lambda_y^S(I)$.
\end{proof}

We now investigate some topological properties of the set $\Gamma_x^S(I)$ of all $I$-statistical cluster points of $x$.

\begin{thm}
Let $A$ be a compact set in $\mathbb R$ and $A\cap {\Gamma}_x^S(I)=\emptyset$. Then the set $\left\{k\in\mathbb N:x_k\in A\right\}$ has $I$-asymptotic density zero.
\end{thm}
\begin{proof}
Since $A\cap {\Gamma}_x^S(I)=\emptyset$, so for any $\xi\in A$ there is a positive number $\varepsilon=\varepsilon(\xi)$ such that
\begin{center}
 $d_I(\left\{k:|x_k-\xi|<\varepsilon\right\})=0$.
\end{center}
Let $B_{\varepsilon(\xi)}(\xi)=\left\{y:|y-\xi|<\varepsilon\right\}$. Then the set of open sets $\left\{B_{\varepsilon(\xi)}(\xi):\xi\in A\right\}$ form an open covers of $A$. Since $A$ is a compact set so there is a finite subcover of  $\{B_{\varepsilon(\xi)}(\xi):\xi\in A\}$ for $A$, say $\left\{A_i=B_{\varepsilon(\xi_i)}(\xi_i):i=1,2,..q\right\}$. Then $A\subset\bigcup\limits_{i=1}^{q} A_i$ and 
\begin{center}
$ d_I(\left\{k:|x_k-\xi_i|<\varepsilon(\xi_i)\right\})=0$ for $i=1,2,...q$.
\end{center}
We can write 
\begin{center}
$\left|\left\{k:k\leq n; x_k\in A\right\}\right|\leq\sum\limits_{i=1}^{q}\left|\left\{k:k\leq n; |x_k-\xi_i|<\varepsilon(\xi_i)\right\}\right|$,
\end{center}
and by the property of $I$-convergence,
\begin{center}
 $
 I \mbox{-}\underset{n \rightarrow \infty}{\lim}\frac{\left|\left\{k:k\leq n; x_k\in A\right\}\right|}{n}
\leq
\sum\limits_{i=1}^{q}  I \mbox{-}\underset{n \rightarrow \infty}{\lim}\frac{|\{k:k\leq n;|x_k-\xi_i|<\varepsilon(\xi_i)\}|}{n}=0.$
\end{center}
Which gives $d_I(\left\{k:x_k\in A\right\})=0$ and this completes the proof.
\end{proof}
\begin{note}
If the set $A$ is not compact then the above result may not be true. 
\end{note}
\begin{thm}
If a sequence $ x=\left\{x_{k}\right\}_{k\in\mathbb N}$ has a bounded $I$-non-thin subsequence, then the set $\Gamma_x^S(I)$ is a non-empty closed set. 
\end{thm}

\begin{proof}
Let $ x=\left\{x_{k_q}\right\}_{q\in \mathbb N}$ is a bounded $I$-non-thin subsequence of $x$ and $A$ be a compact set such that $x_{k_q}\in A$ for each $q\in \mathbb N $. Let $P=\left\{k_q:q\in\mathbb N\right\}$. Clearly $d_I(P)\neq 0$. Now if $\Gamma_x^S(I)=\emptyset$, then $A\cap\Gamma_x^S(I)=\emptyset$ and so by Theorem 3.3 we have 
\begin{center}
$d_I(\{k:x_k\in A\})=0$.
\end{center}
But 
\begin{center}
$\left|\left\{k:k\leq n,~ k\in P\right\}\right|\leq \left|\left\{k:k\leq n,~x_k\in A\right\}\right|$,
\end{center}
which implies that $d_I(P)=0$, which is a contradiction. So $\Gamma_x^S(I)\neq \emptyset$.

 Now to show that $\Gamma_x^S(I)$ is closed, let $\xi$ be a limit point of $\Gamma_x^S(I)$. Then for every $\varepsilon>0$ we have $B_\varepsilon(\xi)\cap (\Gamma_x^S(I)\setminus\{\xi\})\neq\emptyset$. Let $\beta \in B_\varepsilon(\xi)\cap (\Gamma_x^S(I)\setminus\{\xi\})$. Now we can choose $\epsilon'>0$ such that $B_{\epsilon'}(\beta)\subset B_\varepsilon(\xi).$ Since $\beta\in \Gamma_x^S(I)$ so 
 \begin{center}
 $d_I(\{k:|x_k-\beta|<\epsilon'\})\neq\emptyset$
 \end{center}

\begin{center}
$\Rightarrow d_I(\{k:|x_k-\xi|<\varepsilon\})\neq\emptyset$.
\end{center}
Hence $\xi\in\Gamma_x^S(I)$.
\end{proof}
\begin{defn}
(a) A sequence $ x=\left\{x_{k}\right\}_{k\in\mathbb N}$ of real numbers is said to be $ I $- statistically bounded above if, there exists  $L\in\mathbb R$ such that $d_{I} (\{ k \in \mathbb{N}: x_k> L \}) = 0 $.\\
(b) A sequence $ x=\left\{x_{k}\right\}_{k\in\mathbb N}$ of real numbers is said to be $ I $- statistically bounded below if, there exists  $l\in\mathbb R$ such that $d_{I} (\{ k \in \mathbb{N}: x_k< l \}) = 0 $.
\end{defn}
\begin{defn}
A sequence $ x=\left\{x_{k}\right\}_{k\in\mathbb N}$ of real numbers is said to be $ I $- statistically bounded if, there exists  $l > 0$ such that for any $\delta > 0 $, the set 
\begin{center}
$A = \{ n \in \mathbb{N}: \frac{1}{n} |\{ k \in \mathbb{N}: k \leq n, |x_k| > l \}| \geq \delta \} \in I $ 
\end{center}
i.e. $d_{I} (\{ k \in \mathbb{N}: |x_k|> l \}) = 0 $.\\
Equivalently, $ x = \{x_k\}_{k \in \mathbb{N}} $ is said to be $I$-statistically bounded if, there exists a compact set $C$ in $\mathbb R$ such that for any $ \delta > 0 $, the set $ A = \{n\in\mathbb{N}: \frac{1}{n}|\{ k \in\mathbb{N}:k \leq n, x_{k}\notin C \}| \geq \delta\}\in I $ i.e., $d_I(\{ k \in\mathbb{N}: x_{k}\notin C \})=0.$
\end{defn}
 \begin{note}
If $I=I_{fin}=\{A\subset \mathbb N: \left|A\right|<\infty\}$, then the notion of $I$-statistical boundedness coincide with the notion of statistical boundedness.
\end{note}
\begin{cor}
If $ x=\left\{x_{k}\right\}_{k\in\mathbb N}$ is $I$-statistically bounded. Then the set $\Gamma_x^S(I)$ is non empty and compact.
\end{cor}

........................THEOREM 3.3................................
\begin{thm}
 Let $ x=\left\{x_{k}\right\}_{k\in\mathbb N}$ be an $I$-statistically bounded sequence. Then for every $\varepsilon>0$ the set 
 \begin{center}
$\left\{k:d(\Gamma_x^S(I),x_{k})\geq\varepsilon\right\}$ 
 \end{center}
has $I$-asymptotic density zero, where $d(\Gamma_x^S(I),x_{k})=inf_{y\in\Gamma_x^S(I)}|y-x_{k}|$ the distance from $x_{k}$ to the set $\Gamma_x^S(I)$.
\end{thm}
\begin{proof}
Let $C$ be a compact set such that $d_I(\left\{k:x_{k}\notin C\right\})=0$. Then by Corollary 3.5 we have $\Gamma_x^S(I)$ is non-empty and $\Gamma_x^S(I)\subset C$.

Now if possible let $d_I(\left\{k:d(\Gamma_x^S(I),x_{k})\geq\varepsilon'\right\})\neq 0$ for some $\varepsilon'>0$.

 Now we define $B_{\varepsilon'}(\Gamma_x^S(I))=\left\{y:d(\Gamma_x^S(I),y)<\varepsilon'\right\}$ and let $A=C\setminus B_{\varepsilon'}(\Gamma_x^S(I))$. Then $A$ is a compact set which contains an $I$- non-thin subsequence of $x$. Then by Theorem 3.3 $A\cap \Gamma_x^S(I)\neq\emptyset$, which is absurd, since $\Gamma_x^S(I)\subset B_{\varepsilon'}(\Gamma_x^S(I)$. Hence 
 \begin{center}
$d_I(\left\{k:d(\Gamma_x^S(I),x_{k})\geq\varepsilon\right\})=0$ 
 \end{center}
for every $\varepsilon>0$.
\end{proof}
\section{Condition APIO}

\begin{defn}\textbf{(Additive property for $I$-asymptotic density zero sets)}.
The $I$-asymptotic density $d_I$ is said to satisfy APIO if, given a collection of sets $\{A_j\}_{j\in\mathbb N}$ in $\mathbb N$ with $d_I(A_j)=0$, for each $j\in\mathbb N$, there exists a collection $\{B_j\}_{j\in\mathbb N}$ in $\mathbb N$ with the properties $\left|A_j\Delta B_j\right|<\infty$ for each $j\in\mathbb N$ and $d_I(B=\bigcup\limits_{j=1}^\infty B_j)=0$.
\end{defn}
\begin{thm}
A sequence $x=\{x_k\}_{k\in\mathbb N}$ of real number is $I$-statistically convergent to $l$ implies there exists a subset $B$ with $d_I(B)=1$ and $\underset{k\in B, k\rightarrow \infty}{lim}x_k=l$ if and only if $d_I$ has the property APIO.
\end{thm}

\begin{thm}
Let $I$ be an ideal such that $d_I$ has the property APIO, then for any sequence $x=\{x_k\}_{k\in \mathbb N}$ of real numbers there exists a sequence $y=\{y_k\}_{k\in\mathbb N}$ such that $L_y=\Gamma_x^S(I)$ and the set $\{k:x_k\neq y_k\}$ has $I$-asymptotic density zero.
\end{thm}

\section{$I$-statistical analogous of Completeness Theorems }
In this section following the line of Fridy \cite{Fr2}, we formulate and prove an $I$-statistical analogue of the theorems concerning sequences that are equivalent to the completeness of the real line.

We first consider a sequential version of the least upper bound axiom (in $\mathbb R$), namely, Monotone sequence Theorem: every monotone increasing sequence of real numbers which is bounded above, is convergent. The following result is an $I$-statistical analogue of that Theorem.
\begin{thm}
Let $x=\{x_k\}_{k\in\mathbb N}$ be a sequence of real numbers and $M=\{k:x_k\leq x_{k+1}\}$. If $d_I(M)=1$ and $x$ is bounded above on $M$, then $x$ is $I$-statistically convergent.
\end{thm}
\begin{proof}
Since $x$ is bounded above on $M$, so let $l$ be the least upper bound of the range of $\{x_k \}_{k \in M} $. Then we have \\
(i) $x_k\leq l$, $\forall k\in M$\\
(ii) for a pre-assigned $\varepsilon>0$, there exists a natural number $k_0\in M$ such that $x_{k_0}> l-\varepsilon$.\\
Now let $k\in M$ and $k>k_0$. Then $l-\varepsilon<x_{k_0}\leq x_k<l+\varepsilon$. Thus $M\cap\{k:k>k_0\}\subset \{k:l-\varepsilon<x_k<l+\varepsilon\}$. Since the set on the left hand side of the inclusion is of $I$-asymptotic density 1, we have $d_I(\{k:l-\varepsilon<x_k<l+\varepsilon\})=1$ i.e., $d_I(\{k:\left|x_k-l\right|\geq \varepsilon\})=0$. Hence $x$ is $I$-statistically convergent to $l$.
\end{proof}
\begin{thm}
Let $x=\{x_k\}_{k\in\mathbb N}$ be a sequence of real numbers and $M=\{k:x_k\geq x_{k+1}\}$. If $d_I(M)=1$ and $x$ is bounded below on $M$, then $x$ is $I$-statistically convergent.
\end{thm}
\begin{proof}
The proof is similar to that of Theorem 5.1 and so is omitted.
\end{proof}

\begin{note}

(a) In the Theorem 5.1 if we replace the criteria that `$x$ is bounded above on $M$' by `$x$ is $I$-statistically bounded above on $M$' then the result still holds. Indeed if $x$ is a $I$-statistically bounded above on $M$, then there exists $l \in \mathbb{R}$ such that $d_I(\{k\in M: x_k>l\})=0$ i.e., $d_I(\{k\in M:x_k\leq l\})=1$. Let $S=\{k\in M:x_k\leq l\}$ and $l^{'} = sup \{ x_k: k \in S \}$. Then 

(i) $x_k \leq l^{'} $ for all $k \in S$

(ii) for any $\varepsilon > 0 $, there exists a natural number $k_0 \in S $ such that $x_{k_0} > l^{'} - \varepsilon $.
 Then proceeding in a similar way as in Theorem 5.1 we get the result. 

(b) Similarly, In the Theorem 5.2 if we replace the criteria that `$x$ is bounded below on $M$' by `$x$ is $I$-statistically bounded below on $M$' then the result still holds.
\end{note}

Another completeness result for $\mathbb{R}$ is the Bolzano-Weierstrass Theorem, which tells us that, every bounded sequence of real numbers has a cluster point. The following result is an $I$-statistical analogue of that result.
\begin{thm}
Let $I$ be an ideal such that $d_I$ has the property APIO. Let $x=\{x_k\}_{k\in\mathbb N}$ be a sequence of real numbers. If $x$ has a bounded $I$-nonthin subsequence,
then $ x$ has an $I$-statistical cluster point.
\end{thm}
\begin{proof}
Using Theorem 4.2, we have a sequence $y=\{y_k\}_{k\in\mathbb N}$ such that $L_y=\Gamma_x^S(I)$ and $d_I(\{k:x_k=y_k\})=1$. Let $\{x\}_K$ be the bounded $I$-nonthin subsequence of $x$. Then $d_I(\{k:x_k=y_k\}\cap K)\neq 0$. Thus $y$ has a bounded  $I$-nonthin subsequence and hence by  Bolzano-Weierstrass Theorem, $L_y\neq \emptyset$. Thus $\Gamma_x^S(I)\neq\emptyset$. 
\end{proof}
\begin{cor}
Let $I$ be an ideal such that $d_I$ has the property APIO. If $x$ is a bounded sequence of real numbers, then $x$ has an $I$-statistical cluster point.
\end{cor}
 The next result is an $I$-statistical analogue of the Heine-B$\ddot{o}$rel Covering Theorem.
\begin{thm}
Let $I$ be an ideal such that $d_I$ has the property APIO. Let $x$ be a bounded sequence of real numbers, then it has an $I$-thin subsequence $\{x\}_B$ such that $\{x_k:k\in \mathbb N\setminus B\}\cup \Gamma_x^S(I)$ is a compact set.
\end{thm}

\noindent\textbf{Acknowledgment:} 
The second author is grateful to University Grants Commissions, India for his fellowship funding under UGC-JRF (SRF fellowship) scheme during the preparation of this paper.
\\


\begin{thebibliography}{99}

\bibitem{Co1} J. Connor, J. Fridy, and J. Kline, Statistically pre-Cauchy Sequences, \textit{Analysis}, 14(1994) 311-317.
\bibitem{Co2} J. Connor, R-type summability methods, Cauchy criteria, P-sets and Statistical convergence, \textit{Proc. Amer. Math. Soc.} 115 (1992), 319-327.
\bibitem{Co3} J. Connor, The statistical and strong P-Cesaro convergence of sequences, \textit{Analysis}, 8(1988), 47-63.
\bibitem{Da1} P. Das, E. Savas, S.Kr. Ghosal, On generalizations of certain summability methods using ideals, \textit{Appl. Math. Lett.,}
24(2011) 1509-1514.

\bibitem{Da2} P. Das, E. Savas, On $I$-statistically pre-Cauchy sequences, \textit{Taiwanese J. Math}.,18(1)(2014) 115-126.
\bibitem{De} K. Demirci: {\textit{I}-limit superior and limit inferior}, Math. Commun. 6(2)(2001), 165-172.

\bibitem{Fr1} J. A. Fridy, On statistical convergence, \textit{Analysis}, 5(1985), 301-313.
\bibitem{Fr2}  J. A. Fridy, statistical limit points, \textit{Proc. Amer. Math. Soc.} 118(4)(1993), 1187-1192.
\bibitem{Fr3} J. A. Fridy and C. Orhan, Statistical limit superior and limit inferior, \textit{Proc. Amer. Math. Soc.},125(1997), 3625-3631.
\bibitem{La1} B. K. Lahiri, 	P. Das, $I$ and $I^*$-convergence in topological spaces, \textit{Math. Bohemica}, 126(2005), 153-160.
\bibitem{La2} B. K. Lahiri, 	P. Das, $I$ and $I^*$-convergence of nets, \textit{Real Analysis Exchange}, 33(2)(2007/2008), 431-442.
\bibitem{Fa} 	H. Fast, Sur la convergence statistique. \textit{Colloq. Math} 2(1951) 241-244.

\bibitem{Ko1} P. Kostyrko, T. $\check{S}$al$\acute{a}$t, W. Wilczy$\acute{n}$ski: {{\textit{I}}-convergence}, Real Anal.
Exchange 26(2)(2000/2001), 669-685.

\bibitem{Ko2} P. Kostyrko, M. macaz, T. $\check{S}$al$\acute{a}$t, M. Sleziak: {\textit{I}-convergence and external \textit{I}-limit points}, Math. Slovaca 55(4)(2005), 443-454.
\bibitem{Mu} M. Mursaleen, D. Debnath and D. Rakshit, $I$-Statistical Limit Superior and $I$- Statistical Limit Inferior, \textit{Filomat}, 31:7 (2017), 2103–2108.
\bibitem{Ne}  I. Niven and H. S. Zuckerman, An introduction to the theorem of numbers, \textit{4th ed., Wiley,
New York,} 1980.
\bibitem{Pe} S. Pehlivan, A. Guncan and M. A. Mamedov, Statistical cluster points of sequences in finite dimensional spaces, \textit{Czechoslovak Mathematical Journal}, 54 (129)(2004), 95-102.
\bibitem{Sa} T. S\'alat, 	On statistically convergent sequences of real numbers, \textit{Math. Slovaca}, 30(1980), 139-150.
\bibitem{Da3} E. Savas, P. Das, A generalized statistical convergence via ideals, \textit{Appl. Math. Lett.}, 24(2011)
826-830.
\bibitem{Sc} I. J. Schoenberg, The integrability of certain
functions and related summability methods, \textit{ Amer. Math.
Monthly}, 66(1959), 361-375.
\bibitem{St} H. Steinhus, Sur la convergence ordinatre et la convergence asymptotique, \textit{Colloq. Math}., 2(1951) 73-74.





\end{thebibliography}
\end{document}